\documentclass[11pt,leqno]{amsart}

\usepackage[english]{babel}
\usepackage{url}

\usepackage[all,2cell,emtex]{xy}
\usepackage{amssymb,amsmath,bbm,xr,color,stmaryrd}
\usepackage[mathscr]{euscript}
\usepackage{mathrsfs}
\usepackage{tikz-cd}

\usepackage[colorlinks=true]{hyperref} % new hyperref package added by Jean

\DeclareMathOperator{\uK}{\underline K}

\DeclareMathOperator{\uPic}{\underline{Pic}}

\DeclareMathOperator{\rk}{rk}

\DeclareMathOperator{\GW}{GW}

\DeclareMathOperator{\hocolim}{hocolim}

\newcommand{\E}{\mathbb E}

\newcommand{\un}{\mathbbm 1}

\newcommand{\KQ}{\mathrm{KQ}}
\newcommand{\KW}{\mathrm{KW}}
\newcommand{\KGL}{\mathrm{KGL}}

\newcommand{\SL}{\mathrm{SL}}
\newcommand{\pur}{\mathfrak p}

\DeclareMathOperator{\ch}{ch}
\DeclareMathOperator{\borel}{bo}

\newcommand{\vb}[1]{\langle#1\rangle}

\DeclareMathOperator{\SH}{SH}

\DeclareMathOperator{\Hom}{Hom}

\DeclareMathOperator{\uHom}{\underline{Hom}} %Hom interne

\DeclareMathOperator{\CH}{\mathrm{CH}}

\DeclareMathOperator{\spec}{Spec}

\DeclareMathOperator{\Th}{\mathrm{Th}}  

\newcommand{\ZZ} {\mathbb Z}
\newcommand{\QQ} {\mathbb Q}

\renewcommand{\AA} {\mathbb A}

\newcommand{\GG} {\mathbb{G}_m}

\title{Borel isomorphism and absolute purity}

\author{Fr\'ed\'eric D\'eglise}
\address{Institut Math\'ematique de Bourgogne - UMR 5584, Universit\'e de Bourgogne, 9 avenue Alain Savary, BP 47870, 21078 Dijon Cedex, France}
\email{frederic.deglise@ens-lyon.fr}
\urladdr{http://perso.ens-lyon.fr/frederic.deglise/}
 
\author{Jean Fasel}
\address{Institut Fourier - UMR 5582, Universit\'e Grenoble-Alpes, CS 40700, 38058 Grenoble Cedex 9, France}
\email{jean.fasel@gmail.com}
\urladdr{https://www.uni-due.de/~adc301m/staff.uni-duisburg-essen.de/Home.html}

\author{Adeel Khan}
\address{Fakult\"at f\"ur Mathematik, Universit\"at Regensburg, 93040 Regensburg, Germany}
\email{adeel.khan@mathematik.uni-regensburg.de}
\urladdr{https://www.preschema.com}

\author{Fangzhou Jin}
\address{Fakult\"at f\"ur Mathematik,  
Universit\"at Duisburg-Essen, Campus Essen,
Thea-Leymann-Strasse 9,
45127 Essen}
\email{fangzhou.jin@uni-due.de}
\urladdr{https://sites.google.com/site/fangzhoujin1/}

\date{\today}

\newtheorem{thm}{Theorem}[section]
\newtheorem{prop}[thm]{Proposition}
\newtheorem{lm}[thm]{Lemma}
\newtheorem{cor}[thm]{Corollary}

\theoremstyle{remark} % env. "remark" de l'AMS
\newtheorem{rem}[thm]{Remark}

\theoremstyle{definition} % env. "definition" de l'AMS
\newtheorem{df}[thm]{Definition}
\newtheorem{num}[thm]{}

\numberwithin{equation}{thm}

\newtheorem{thm*}{Theorem}

\begin{document}

\begin{abstract}
We prove absolute purity for the rational motivic sphere spectrum.
The main ingredient is the construction of an analogue of the Chern character, where algebraic K-theory is replaced by hermitian K-theory, and motivic cohomology by the plus and minus parts of the rational sphere spectrum.
Another ingredient is absolute purity for hermitian K-theory.
\end{abstract}

\maketitle

\setcounter{tocdepth}{3}
\tableofcontents

\subsection*{Introduction}

The \emph{absolute purity conjecture}, stated for \'etale torsion sheaves
 and by extension for $\ell$-adic sheaves, has been a difficult problem
 since its formulation by Grothendieck in the mid-sixties (published in 1977
 \cite{SGA5}). For some time, only the case of one-dimensional regular schemes was
 known thanks to Deligne, until Thomason first solved the case of $\ell$-adic sheaves
 (\cite{Thomason1}).
 His proof was later extended by Gabber to the general case (see \cite{Fuji}).
 An ultimate proof was found by Gabber, using a refinement of De Jong resolution
 of singularities, published in \cite[Exp.~XVI]{Gabber}.

The importance of this conjecture stands from its applications. First, it allows to show
 that constructibility (of complexes of \'etale sheaves) is stable under the direct
 image functor $f_*$ (for $f$ of finite type between quasi-excellent schemes).
 One deduces that constructibility is stable under the six operations (under very general
 assumptions). Then one obtains the so-called Grothendieck--Verdier duality
 for constructible complexes over schemes $S$ with a dimension function.\footnote{Most
 notably,
 the existence of a dualizing complex $\mathcal D_S$ over $S$ such that
 $D_S=\uHom(-,\mathcal D_S)$ is an auto-equivalence of categories. The functor $D_S$
 then transforms $f_*$ (resp. $f^*$) into $f_!$ (resp. $f^!$).} This last point implies
 the existence of the (self-dual) perverse $t$-structure over suitable base schemes,
 extending the fundamental work of \cite{BBD} (see \cite{Gabber}).

For \emph{triangulated mixed motives}, modeled on the previous \'etale setting by
 Beilinson, this conjecture was implicit in the expected property.
 It was first formulated and proved in the rational case by Cisinski and the first-named
 author in \cite{CD3}. Later the \emph{absolute purity property} was explicitly
 highlighted in \cite[Appendix]{CD5}, and proven for integral \'etale motives
 in \cite{AyoubEt} and \cite{CD5}. It became apparent that this important property
 should hold in greater generality, and philosophically be an addition to the
 six functors formalism. Thus, it was conjectured in \cite{Deg12} that this property should hold for the algebraic cobordism spectrum and the sphere spectrum
 of Morel and Voevodsky's motivic homotopy theory.
 Up to this point, the main evidence for this conjecture was the example of the algebraic K-theory spectrum, for which absolute purity was proven in \cite[Thm.~13.6.3]{CD3}.

\bigskip

The aim of this note is to prove the absolute purity conjecture for
 the rational motivic sphere spectrum. Recall that the latter is known to split into ``plus'' and ``minus'' parts \cite[Sect.~16.2]{CD3}. It was established in \emph{loc. cit.} that the
 plus part agrees with the rational motivic cohomology spectrum.
 Moreover, the latter was proven to satisfy the absolute purity property, by using the Chern character to reduce to the case of the algebraic K-theory spectrum $\KGL$.
 Indeed, absolute purity for the latter is a consequence of Quillen's d\'evissage theorem, in the form of an equivalence $K(X~\mathrm{on}~Z) \simeq K(Z)$ for a closed immersion of regular schemes $i : Z \to X$.
 To prove absolute purity for $\un_\QQ$ itself, we employ a similar strategy, replacing Quillen's algebraic K-theory by hermitian K-theory.
 That the latter satisfies absolute purity is a consequence of the foundational work of Karoubi, Hornbostel and Schlichting, through which hermitian K-theory is known to have all the expected properties, especially the analogue of Quillen's d\'evissage. That being given, the crucial new ingredient
 in our proof is the construction of an appropriate analogue of the Chern character for
 hermitian K-theory over nice base schemes (defined over $\ZZ[1/2]$).
 We call it the \emph{Borel character}.
 Its formulation allows the computation of the rational hermitian K-theory spectrum
 $\KQ_\QQ$ in terms of the plus and minus parts of the rational sphere spectrum
 $\QQ_{S+}$ and $\QQ_{S-}$, which play the role of the rational motivic cohomology spectrum.
 With these notations, the Borel character is an isomorphism of the ring spectra:
$$
\borel:\KQ_\QQ\to \bigoplus_{m \in \ZZ} \QQ_{S+}(2m)[4m]
 \oplus \bigoplus_{m \in \ZZ} \QQ_{S-}(4m)[8m]
$$
(see Definition \ref{df:Borel}).
 The absolute purity conjecture for the rational sphere spectrum is then deduced
 from the analogous property established for hermitian K-theory:
 see Theorem~\ref{thm:abs_purity} and Corollary~\ref{cor:absolute purity examples}. An interesting application
 is the existence of a well-defined product for rational Chow--Witt groups
 of a regular base $\ZZ[\frac 12]$-scheme.

 Our construction of the Borel character uses in an essential way previous works of Ananyevskiy
 \cite{Anan1}
 and Ananyevskiy, Levine, Panin \cite{ALP}. The Borel character will be studied
 further in \cite{DFJK1}, where an explicit formula in terms of characteristic classes will be given over a base field.

A very noticeable consequence of the Borel isomorphism
 is that every rational motivic spectrum, over general base schemes, is \emph{Sp-oriented} in the
 sense of Panin and Walter \cite{PaninWalter1}. This is the
 $\AA^1$-homotopy analogue of the well-known fact that every rational spectrum in topology
 is oriented. Note that in $\AA^1$-homotopy, there exist rational ring spectra
 that are non-orientable in the classical sense (say, which do not admit Chern classes): for example, consider Chow--Witt groups and hermitian K-theory, rationally, over fields with non-trivial Grothendieck--Witt groups.

\subsection*{Organization of the paper} The paper is divided into three sections.
 In Section 1, we give some quick reminders on some  ring spectra, such as periodicity
 and representability of hermitian K-theory and Balmer's higher Witt groups
 (for regular schemes). In Section 2, we construct the Borel isomorphism
 and deduce that every rational spectrum is Sp-orientable.
 In Section 3, we establish the absolute purity of the rational sphere spectrum
 and draw some consequences.

\subsection*{Conventions}

All schemes are noetherian and finite dimensional, admit an ample family of line bundles\footnote{This assumption allows us to use Schlichting's results in \cite{Schlicht1}.  It could be avoided by using perfect complexes instead of strictly perfect complexes in \emph{op. cit}.},
 and are defined over $\ZZ[1/2]$.

\bigskip

\noindent \textbf{Acknowledgments.}
F. D\'eglise received support from the French ``Investissements d'Avenir'' program, project ISITE-BFC (contract ANR-lS-IDEX-OOOB).

\section{Basics on hermitian K-theory and higher Witt groups}

\begin{df}
Let $S$ be a scheme as in our conventions.
We will denote by $\KQ_S$ the motivic ring spectrum representing hermitian K-theory
 over $S$, denoted by $\mathbf{BO}$ in \cite{PaninWalter1}.
\end{df}
We need the following properties:
\begin{itemize}
\item[(GW1)] 
(\cite[Th. 1.2]{PaninWalter1}) 
Given any map $f:T \rightarrow S$, there is a canonical identification $f^*\KQ_S=\KQ_T$. In other words, $\KQ$
 is an absolute ring spectrum.
 This follows from the geometric model of hermitian K-theory using quaternionic
 Grassmanians.
\item[(GW2)]
(\cite[Cor.~7.3]{PaninWalter1})
For any regular scheme $S$ and any closed immersion $i : Z \to S$, there are isomorphisms:
\begin{equation}\label{eq:schlicht}
\KQ^{n,i}(S/S-Z)=\GW_{2i-n}^{[i]}(S \text{ on } Z)
\end{equation}
for all pairs $(n,i)\in\ZZ\times\ZZ$.
Here the right-hand side is Schlichting's higher Grothendieck--Witt groups:
 that is, the $(2i-n)$-th homotopy group of the spectrum
 $\mathbb GW^{[i]}(\mathcal A_S \text{ on } Z,\mathcal O_S)$ of \cite[Def. 8 of Section 10]{Schlicht1}.
 In the following, we will simply denote this spectrum by $\mathbb GW^{[i]}(S \text{ on } Z)$.
\end{itemize}

\begin{rem}
Under the twisting notation introduced for example in \cite{DJK},
 one can rewrite \eqref{eq:schlicht} as: $\KQ^{n}(S/S-Z, \vb{i})=\GW_{n}^{[i]}(S~\text{on}~Z)$.
\end{rem}

The following result is well-known (see e.g. \cite{GepSna}):
\begin{prop}
Let $\E$ be a motivic ring spectrum over $S$.
 Consider a pair of integers $(n,i) \in \ZZ^2$.
 Then the following conditions are equivalent:
\begin{enumerate}
\item There exists an element $\rho \in \E_{n,i}(S)$, invertible for the cup-product
 on $\E_{**}$.
\item There exists an isomorphism: $\tilde \rho:\E(i)[n] \rightarrow \E$.
\end{enumerate}
\end{prop}

\begin{df}
A pair $(\E,\rho)$ satisfying the equivalent conditions of the above proposition will be called an \emph{$(n,i)$-periodic ring spectrum} over $S$.

An \emph{absolute} $(n,i)$-periodic ring spectrum is an $(n,i)$-periodic ring spectrum over $\spec(\ZZ[\frac 12])$.\footnote{The restriction to $\spec{\ZZ[\frac 12]}$ comes
 from the global conventions of our paper. Needless to say the definition makes
 sense over $\spec{\ZZ}$.}
\end{df}

\begin{prop}
There exists a family of elements $\rho_S \in \KQ^{8,4}(S)$ indexed by schemes,
 stable under pullback, such that $(\KQ_S,\rho_S)$ is $(8,4)$-periodic.
\end{prop}
This follows from the construction of the spectrum $\KQ_S$.
 The element $\rho_S$ can be defined using \cite[Prop. 7]{Schlicht1},
 which implies that there exists a canonical isomorphism of spectra:
$$
\mathbb GW^{[0]}(S) \simeq \mathbb GW^{[4]}(S).
$$
Therefore using (GW1),
 one gets an isomorphism: $\psi_S:\KQ^{0,0}(S) \xrightarrow{\sim} \KQ^{8,4}(S)$
 and we can put $\rho_S=\psi_S(1)$.

Following \cite{Anan1}, we introduce the following $\eta$-periodic spectra.
\begin{df}[Ananyevskiy]\label{df:KW}
Let $\eta:\un_S \rightarrow \un_S(-1)[-1]$ be the (desuspended) Hopf map.
 We define the $\eta$-periodized sphere spectrum $\un_S[\eta^{-1}]$ as:
$$
\un_S[\eta^{-1}]=\hocolim\big( \un_S \xrightarrow{\eta} \un_S(-1)[-1]
  \xrightarrow{\eta(-1)[-1]} \un_S(-2)[-2]  \xrightarrow{\eta(-2)[-2]} \hdots \big).
$$
Given any spectrum $\E$, we set $\E\wedge \un_S[\eta^{-1}]=\E[\eta^{-1}]$. In the special case of hermitian $K$-theory, we set $\KW_S=\KQ_S[\eta^{-1}]$. This defines an absolute ring spectrum.
\end{df}
In other words, $(\KW_S,\eta)$ is $(1,1)$-periodic.
 It is in fact the $(1,1)$-periodization of $\KQ_S$. Note also that the element
 $\rho_S \in \KQ^{8,4}(S)$ induces an element still denoted by $\rho_S \in \KW^{8,4}(S)$,
 and the above definition shows that $(\KW_S,\rho_S)$ is $(8,4)$-periodic.

Recall the following result of Ananyevskiy:

\begin{thm}[Ananyevskiy]
For any regular scheme $S$, there exists an isomorphism:
$$
\KW^{n,i}(S) \simeq W^{[i-n]}(S)
$$
where the right-hand side is Balmer's higher Witt groups.
\end{thm}

This is stated in \cite[Theorem~6.5]{Anan1} in the special case of smooth varieties, but the same proof applies here.
This isomorphism is contravariantly functorial in $S$, and induces an isomorphism of bigraded rings.

\begin{rem}
Over non-regular schemes, $\KQ$ and $\KW$ represent the $\AA^1$-invariant versions of higher Grothendieck--Witt and higher Witt groups, respectively.
\end{rem}

\section{Rational Borel isomorphism}

\begin{num}
As $\KW$ is $\eta$-periodic, the unit map $\un_S \rightarrow \KW_S$
 uniquely factors through a map
\begin{equation}
\label{eq:unit_KW}
\un_S[\eta^{-1}]\rightarrow \KW_S.
\end{equation}
The uniqueness of the factorization ensures that the latter map
 is a morphism of ring spectra.
\end{num}

\begin{num}\label{num:Morel_decomposition}
We fix an arbitrary base scheme $S$.
The symmetry involution permuting the factors  $\GG \wedge \GG$ induces an
 involution $\epsilon:\un_S \rightarrow \un_S$. We have two complementary projectors
 on $\un_S[1/2]$:
$$
e_+=\frac{1-\epsilon} 2, \phantom{i} e_-=\frac{1+\epsilon} 2,
$$
yielding Morel's decomposition: $\un_S[1/2]=\un_{S+} \oplus \un_{S-}$.
 More generally, given any spectrum $\E$ over $S$, we get a canonical decomposition:
$$
\E[1/2]=\E_+ \oplus \E_-
$$
such that $\epsilon$ acts by $+1$ (resp. $-1$) on $E_+$ (resp. $\E_-$).

Recall from Morel's computation that one has $\eta=\epsilon\eta$.
 In particular, we get:
$$
\un_S[1/2,\eta^{-1}]=\un_{S-}.
$$
In view of Definition \ref{df:KW}, we then deduce:
$$
\KQ_- \simeq \KW[1/2]=\KW_-.
$$

Recall that $(\KW_S,\rho_s)$ is $(8,4)$-periodic.
 One deduces a canonical map:
$$
\bigoplus_{m \in \ZZ} \un_S(4m)[8m] \xrightarrow{\sum_m \rho_S^m} \KW_S.
$$
Taking the rational parts and projecting this map to the minus part,
 we finally obtain a canonical map, uniquely determined by $\rho_S$:
$$
\psi_S:\bigoplus_{m \in \ZZ} \QQ_{S-}(4m)[8m] \xrightarrow{\sum_m \rho_S^m} \KW_{S,\QQ-}.
$$
Note that by construction, the maps $\psi_S$ are compatible with pullbacks in $S$.
 The following result follows from \cite[Corollary~3.5]{ALP}.
\end{num}
\begin{thm}\label{thm:rational_witt}
For any scheme $S$, the map $\psi_S$ is an isomorphism.
\end{thm}
\begin{proof}
Since the formation of the map $\psi_S$ is compatible with base change, by Lemma~\ref{lm:point_conservative} below we are reduced to the case where $S$ is the spectrum of a field (of characteristic different from $2$, since $S$ is a $\ZZ[1/2]$-scheme by our conventions).
In this case the result follows from \cite[Corollary~3.5]{ALP}.
\end{proof}

The following standard result, used in the proof above, does not require the assumptions of our global conventions.

\begin{lm}
\label{lm:point_conservative}
Let $S$ be a noetherian scheme. For any point $x$ of $S$, denote by $i_x:\spec{\kappa(x)} \rightarrow S$ the inclusion of the spectrum of the residue field. Then the family of functors $i_x^*:\SH(S)\to\SH(\kappa(x))$, for all points $x\in S$, is conservative.
That is, a map $f$ in $\SH(S)$ is an isomorphism if and only if for every point $x\in S$, the map $i_x^*(f)$ is an isomorphism in $\SH(\kappa(x))$.
\end{lm}

\begin{proof}
We use the following notation: if $f:X\to Y$ is a morphism of schemes and $P\in\SH(Y)$, we denote by $P_{|X}$ the object $f^*P\in\SH(X)$.

It suffices to prove that if $A\in\SH(S)$ is an object such that $A_{|x}\in\SH(\kappa(x))$ is $0$ for all points $x\in S$, then $A=0$. We proceed by noetherian induction on the scheme $S$. Suppose that the claim holds for every proper closed subscheme of $S$. Let $A\in\SH(S)$ such that $A_{|x}\in\SH(\kappa(x))$ is $0$ for all points $x\in S$. Let $\eta$ be a generic point of $S$, then $A_{|\eta}=0$ by hypothesis. Since $\eta=\varprojlim(U)$ where $U$ runs over all open neighborhoods of $\eta$ in $S$, it follows by continuity of $\SH$ that there exists a non-empty open subscheme $U$ of $S$ such that $A_{|U}=0$. Denote by $j:U\to S$ the open immersion, with reduced closed complement $i:Z\to S$. By the localization theorem we have a distinguished triangle in $\SH(S)$:
$$
j_!A_{|U}\to A\to i_*A_{|Z}\to j_!A_{|U}[1].
$$
By assumption and the induction hypothesis, $A_{|U}=0$ and $A_{|Z}=0$, which implies that $A=0$ in $\SH(S)$, and the result follows.
\end{proof}

\begin{df}
We denote by
  \begin{equation*}
    \borel_{S,-} : \KW_{S,\QQ-} \to \bigoplus_{m \in \ZZ} \QQ_{S-}(4m)[8m].
  \end{equation*}
the isomorphism of ring spectra inverse to $\psi_S$.
\end{df}

Note the following remarkable corollary, due to the fact that $\KQ$ is Sp-oriented.
\begin{cor}\label{cor:sp-orinted}
For any scheme $S$, any rational ring spectrum over $S$ admits a canonical
 $\mathrm{Sp}$-orientation.
 Indeed, $\un_{\QQ,S}$ is the universal $\mathrm{Sp}$-orientable ring spectrum over $S$.
In particular, the Thom space functor factors through Deligne's Picard functor
 as follows:
$$
\xymatrix@R=15pt@C=30pt{
\uK(S)\ar^{\Th_{S,\QQ}}[r]\ar_{(\det,\rk)}[d] & \SH(S)_\QQ^{\otimes} \\
\uPic(S)\ar@{-->}_{\Th'_{S,\QQ}}[ru]
}
$$
In particular, the rational stable Thom space of a vector bundle depends only
 on its determinant and its rank.
\end{cor}

Using \cite{DJK} and the Sp-orientability from the previous corollary,
 one deduces the following result.
\begin{cor}\label{cor:fdl}
Let $\E$ be an arbitrary rational ring spectrum $\E$ over $S$.
 For a pair of integers $(n,i) \in \ZZ^2$, a $S$-scheme $X$ with structural map $p$
 (resp. $p$ separated of finite type),
 and a line bundle $L$ over $X$, one puts:
\begin{align*}
\E^{n,r}(X,L)&=\Hom_{\SH(S)}(\un_X,p^*\E(r)[n] \otimes \Th_S(L)), \\
\text{resp. } \E_{n,r}(X/S,L)&=\Hom_{\SH(S)}(\un_X(r)[n] \otimes \Th_S(L),p^!\E).
\end{align*}
Then for any smoothable lci morphism
  $f:X \rightarrow S$ of relative virtual dimension $r$,
 there exists a fundamental class $\eta_f \in \E_{n,r}(X/S,\det{L_f})$.
 Altogether, these fundamental classes satisfy
 compatibility with composition and the excess intersection formula.
\end{cor}

\begin{num}
Let again $S$ be an arbitrary scheme.
 Recall from \cite[Th. 3.4]{RonOst} that one has a canonical distinguished triangle:
$$
\KQ_S(1)[1] \xrightarrow{\eta} \KQ_S \xrightarrow f \KGL_S \rightarrow \KQ_S(1)[2]
$$
where $\KGL$ is the spectrum representing the homotopy invariant $K$-theory over $S$
 and $f$ the \emph{forgetful map}.

As $\eta_+=0$ and $\KGL_{S-}=0$, we immediately deduce the following result.
\end{num}
\begin{prop}
One has a split exact sequence in $\SH(S)[1/2]$, and more precisely in $\SH(S)_+$,
 the essential image of the projector $e_+$ of Paragraph \ref{num:Morel_decomposition}:
$$
0 \rightarrow \KQ_{S+}
 \xrightarrow f \KGL_S[1/2] \rightarrow \KQ_{S+}(1)[2] \rightarrow 0.
$$
In other words, $\KGL_S[1/2] \simeq \KQ_{S+} \oplus \KQ_{S+}(1)[2]$.
\end{prop}
There is moreover a canonical splitting of the above map. Indeed, consider the ``twisted'' forgetful map
\[
f(1)[2]: \KQ_{S,+}(1)[2]\to  \KGL_S(1)[2]\simeq  \KGL_S
\]
Then, the composite 
\[
\KQ_{S,+}(1)[2]\xrightarrow{f(1)[2]}  \KGL_S\xrightarrow{H}  \KQ_{S+}(1)[2]
\]
is just multiplication by $2$ and $\frac 12f(1)[2]$ is a section of $H$.

Recall from \cite{Riou, CD3} that the classical Chern character corresponds to an
 isomorphism of the following form in $\SH(S)$:
$$
\ch:\KGL_{S,\QQ} \rightarrow \bigoplus_{m \in \ZZ} \QQ_{S+}(m)[2m]
$$
where $\QQ_{S+}$ is identified with the rational motivic Eilenberg-MacLane spectrum
 (equivalently, the universal orientable ring spectrum).
\begin{prop}
The composition
$$
\KQ_{S,\QQ+} \xrightarrow f \KGL_{\QQ,S}
 \xrightarrow{\ch} \bigoplus_{m \in \ZZ} \QQ_{S+}(m)[2m]
$$
induces an isomorphism
$$
\KQ_{S,\QQ+} \xrightarrow{\borel_{S,+}} \bigoplus_{m \in \ZZ} \QQ_{S+}(2m)[4m].
$$ 
\end{prop}
\begin{proof}
According to \cite{Schlicht1},
 there is an isomorphism $\mathbb GW^{[0]} \simeq \mathbb GW^{[2]}_\epsilon$
 where $\epsilon$ consists in taking the opposite duality.
 In particular, we get an isomorphism of functors
 $\mathbb GW^{[0]} \simeq \mathbb GW^{[4]}$ and, using the isomorphism \eqref{eq:schlicht},
 one deduces there is an element $\sigma_S \in \KQ_+^{4,2}(S)$
 such that $(\KQ_+,\sigma_S)$ is $(4,2)$-periodic.
 By construction, one has $\sigma_S^2=\rho_S$ and one can check that $f(\sigma_S)=\beta^2$,
 where $\beta$ is the Bott element in K-theory (expressing its (2,1)-periodicity).
 This finishes the proof.
\end{proof}

\begin{num}
In particular, one gets a canonical isomorphism:
\begin{equation}\label{eq:pontryagin}
\begin{split}
\KQ_{S,\QQ} 
\simeq 
&\KQ_{S,\QQ+} \oplus \KQ_{S,\QQ-}\\
 \xrightarrow{\borel_{S,+} \oplus \borel_{S,-}} 
&\bigoplus_{m \in \ZZ} \QQ_{S+}(2m)[4m]
 \oplus \bigoplus_{m \in \ZZ} \QQ_{S-}(4m)[8m],
\end{split}
\end{equation}
which, from the above constructions, is in fact an isomorphism of ring spectra.
\end{num}
\begin{df}\label{df:Borel}
We call the above isomorphism the \emph{Borel character} and denote it by 
\[
\borel_S:\KQ_{S,\QQ}\to \bigoplus_{m \in \ZZ} \QQ_{S+}(2m)[4m]
 \oplus \bigoplus_{m \in \ZZ} \QQ_{S-}(4m)[8m].
\]
\end{df}

 \begin{rem}
In case $S$ is a perfect field, it is possible to give to the above Borel character a form which
 is closer to the classical Chern character, if we think of its target in terms
 of motivic and MW-motivic cohomology. We refer the reader to our future work
 \cite{DFJK1} for more details.
\end{rem}

\section{Absolute purity}

\begin{num}
Let $\E$ be an absolute spectrum over a scheme $S$.
 Recall that $\E$ is said to satisfy \emph{absolute purity} (see \cite[4.3.11]{DJK})
 if, for any regular schemes $X$ and $Y$ and any smoothable (hence lci) morphism $f:Y \rightarrow X$ with cotangent complex $L_f$,
 the following purity transformation is an isomorphism:
\begin{equation}
\label{eq:pur_trans}
\pur_f:\E_Y \otimes \Th_Y(\vb{L_f}) \rightarrow f^!(\E_X).
\end{equation}

This property is stable under retracts and tensor products with strongly dualizable objects \cite[Remark~4.3.8(iii)]{DJK}.
\end{num}

In the setting of Corollary~\ref{cor:fdl}, a pleasant consequence of the absolute purity property is the following duality statement:

 \begin{cor}\label{cor:duality}
Consider an arbitrary rational ring spectrum $\E$ over $S$,
 and adopt the notations of Corollary \ref{cor:fdl}.
If $\E$ satisfies absolute purity, then for any smoothable morphism
 $f:X \rightarrow S$ between regular schemes with cotangent complex $L_f$,
 the following map is an isomorphism:
$$
\E^{n,r}(X,L) \rightarrow \E_{n,r}(X/S,\det(L_f)-L), x \mapsto x.\eta_f.
$$

\end{cor}

We have the following result (compare \cite[Thm.~13.6.3]{CD3}):
\begin{thm}\label{thm:abs_purity}
The absolute spectrum $\KQ$ over $\spec(\ZZ[1/2])$ satisfies absolute purity.
\end{thm}
\begin{proof}
It suffices to deal with the case where $i:Z\to X$ is a closed immersion of constant codimension $c$ between regular schemes. As the functor $i_*:\SH(Z)\to\SH(X)$ is conservative, we are reduced to show that the map
\begin{equation}
\label{eq:i_*fdl}
i_*(\KQ_Z\otimes\Th_Z(-N_i))
\xrightarrow{}
i_*i^!\KQ_X
\end{equation}
induced by the purity transformation~\eqref{eq:pur_trans}, is an isomorphism. It suffices to show that for any smooth $X$-scheme $T$ and any integers $m,n\in\mathbb{Z}$, the map
\begin{equation}
\label{eq:sigma_KQ}
[\Sigma^\infty_XT_+(m)[n],i_*(\KQ_Z\otimes\Th_Z(-N_i))]
\xrightarrow{}
[\Sigma^\infty_XT_+(m)[n],i_*i^!\KQ_X]
\end{equation}
obtained by applying the functor $[\Sigma^\infty_XT_+(m)[n],\cdot]$ to the map~\eqref{eq:i_*fdl}, is an isomorphism. Denote by $T_Z=Z\times_XT$, which is a closed subscheme of $T$ and is smooth over $Z$.
By~\eqref{eq:schlicht}, since the spectrum $\KQ$ is $\SL^c$-oriented (\cite[7.4]{PaninWalter1}), the left-hand side of~\eqref{eq:sigma_KQ} is computed as
\begin{equation*}
  [\Sigma^\infty_ZT_{Z+}(m)[n],\KQ_Z\otimes\Th_Z(-N_i)]
    \simeq \GW^{[-m-c]}_{2m-n}(T_Z,\det(N_i)).
\end{equation*}
On the other hand, by~\eqref{eq:schlicht}, the localization sequence, and \cite[9.5]{Schlicht2}, the right-hand side of~\eqref{eq:sigma_KQ} is computed as
$$
[\Sigma^\infty_XT_+(m)[n],i_*i^!\KQ_X]
\simeq
\GW^{[-m]}_{2m-n}(T\mathrm{\ on\ } T_Z).
$$
Moreover, the map \eqref{eq:sigma_KQ} is identified under these identifications with the Gysin map in Grothendieck--Witt theory induced by direct image of coherent sheaves (see \cite[(9.9)]{Schlicht2}).
The result then follows from the d\'evissage theorem for Grothendieck--Witt groups \cite[Theorems~9.5, 9.18 and 9.19]{Schlicht2} (which is analogous to \cite[Proposition 28]{FaselSrinivas} and \cite{GilleAM}).
\end{proof}

\begin{cor}\label{cor:absolute purity examples}
The following absolute spectra over $\spec(\ZZ[1/2])$ satisfy absolute purity:
\begin{enumerate}
\item the Witt ring spectrum $\KW$;
\item the rational sphere spectrum $\un_\QQ$;
\item any strongly dualizable rational ring spectrum;
\item any cellular rational spectrum in the sense of Dugger--Isaksen
 (see \cite[2.10]{DI_cell})
\end{enumerate}
\end{cor}
The case of $\KW$ is clear from definition \ref{df:KW},
 as $i^* \otimes \Th(-N_i)$ and $i^!$ commute with homotopy colimits
 (for $i^!$ we apply the localization property).
For $\un_\QQ \simeq \QQ_{S+} \oplus \QQ_{S-}$, note that the Borel isomorphism
 \eqref{eq:pontryagin} exhibits both of its summands as retracts of $\KQ_S$.
 The other cases follow formally.

\begin{rem}
In particular, the duality statement of Corollary~\ref{cor:duality} applies to all of the above examples.
\end{rem}

The absolute purity property has interesting applications
 for Chow--Witt groups of regular schemes (see \cite{FaselSrinivas} for their
 definition without a base field).
 Based on the method of proof of \cite[4.2.6]{Deglise16}, we get:
\begin{cor}
Let $S$ be a regular scheme.
 Then for any integer $n \geq 0$, there exists an isomorphism:
$$
H^{2n,n}_{\AA^1}(S,\QQ) \simeq \widetilde{\CH}^n(S) \otimes_\ZZ \QQ
$$
where the right hand-side is the Chow--Witt group of $S$ with coefficients in $\QQ$.
As a consequence, rational Chow--Witt groups when restriced to regular schemes
 admit products and Gysin maps with respect to projective morphisms.
\end{cor}
This follows from the hyper-cohomology spectral sequence with respect to the 
 $\delta$-homotopy $t$-structure. Gysin morphisms follow from the construction
 of \cite{DJK}.

Further, we can also deduce comparison results for certain singular schemes.
For the definition of Chow--Witt groups of singular schemes, with a dimension function
 and a dualizing sheaf, we refer the reader to \cite{Gille}.
\begin{cor}
Let $S$ be a regular scheme
 and $X$ be an $S$-scheme essentially of finite type.
 Then for any integer $n \geq 0$, there exists an isomorphism:
$$
H_{2n,n}^{\AA^1}(X/S,\QQ) \simeq \widetilde{\CH}_{\delta=n}(X) \otimes_\ZZ \QQ
$$
where the right hand-side is the Chow--Witt group of $S$ 
 of quadratic cycles sums of points $x$ such that $\delta(x)=n$, tensored with $\QQ$.
As a consequence, these groups admit Gysin maps for
 smoothable lci morphisms of $S$-schemes essentially of finite type.
\end{cor}
This follows from the hyper-homology spectral sequence with respect to the 
 $\delta$-homotopy $t$-structure.

\bibliographystyle{amsalpha}
\bibliography{purity}

\providecommand{\bysame}{\leavevmode\hbox to3em{\hrulefill}\thinspace}
\providecommand{\MR}{\relax\ifhmode\unskip\space\fi MR }
\providecommand{\MRhref}[2]{%
  \href{http://www.ams.org/mathscinet-getitem?mr=#1}{#2}
}
\providecommand{\href}[2]{#2}
\begin{thebibliography}{BBD82}

\bibitem[ALP17]{ALP}
A.~Ananyevskiy, M.~Levine, and I.~Panin, \emph{Witt sheaves and the
  {$\eta$}-inverted sphere spectrum}, J. Topol. \textbf{10} (2017), no.~2,
  370--385. \MR{3653315}

\bibitem[Ana16]{Anan1}
A.~Ananyevskiy, \emph{On the relation of special linear algebraic cobordism to
  {W}itt groups}, Homology Homotopy Appl. \textbf{18} (2016), no.~1, 204--230.
  \MR{3491850}

\bibitem[Ayo14]{AyoubEt}
J.~Ayoub, \emph{La r\'{e}alisation \'{e}tale et les op\'{e}rations de
  {G}rothendieck}, Ann. Sci. \'{E}c. Norm. Sup\'{e}r. (4) \textbf{47} (2014),
  no.~1, 1--145. \MR{3205601}

\bibitem[BBD82]{BBD}
\emph{Analyse et topologie sur les espaces singuliers. {I}}, Ast\'{e}risque,
  vol. 100, Soci\'{e}t\'{e} Math\'{e}matique de France, Paris, 1982.
  \MR{751965}

\bibitem[CD09]{CD3}
D.-C. Cisinski and F.~D{\'e}glise, \emph{Triangulated categories of mixed
  motives}, arXiv:0912.2110, version 3, 2009.

\bibitem[CD15]{CD5}
\bysame, \emph{Integral mixed motives in equal characteristics}, Doc. Math.
  (2015), no.~Extra volume: Alexander S. Merkurjev's sixtieth birthday,
  145--194.

\bibitem[D{\'e}g19]{Deg12}
F.~D{\'e}glise, \emph{Orientation theory in arithmetic geometry}, to appear in
  Acts of "International Colloquium on K-theory", TIFR, 2019.

\bibitem[DF16]{Deglise16}
F.~D{\'e}glise and J.~Fasel, \emph{{MW}-motivic complexes}, arXiv:1708.06095,
  2016.

\bibitem[DFKJ]{DFJK1}
F.~D\'eglise, J.~Fasel, A.~Khan, and F.~Jin, \emph{The {B}orel character}, in
  preparation.

\bibitem[DI05]{DI_cell}
D.~Dugger and D.~C. Isaksen, \emph{Motivic cell structures}, Algebr. Geom.
  Topol. \textbf{5} (2005), 615--652. \MR{2153114}

\bibitem[DJK18]{DJK}
F.~D{\'e}glise, F.~Jin, and A.~Khan, \emph{Fundamental classes in motivic
  homotopy theory}, arXiv:1805.05920, 2018.

\bibitem[FS09]{FaselSrinivas}
J.~Fasel and V.~Srinivas, \emph{Chow-{W}itt groups and {G}rothendieck-{W}itt
  groups of regular schemes}, Adv. Math. \textbf{221} (2009), no.~1, 302--329.

\bibitem[Fuj02]{Fuji}
K.~Fujiwara, \emph{A proof of the absolute purity conjecture (after {G}abber)},
  Algebraic geometry 2000, {A}zumino ({H}otaka), Adv. Stud. Pure Math.,
  vol.~36, Math. Soc. Japan, Tokyo, 2002, pp.~153--183. \MR{1971516}

\bibitem[Gil07a]{GilleAM}
S.~Gille, \emph{The general d\'evissage theorem for witt groups of schemes},
  Arch. Math. \textbf{88} (2007), 333--343.

\bibitem[Gil07b]{Gille}
\bysame, \emph{A graded {G}ersten-{W}itt complex for schemes with a dualizing
  complex and the {C}how group}, J. Pure Appl. Algebra \textbf{208} (2007),
  no.~2, 391--419.

\bibitem[Gro77]{SGA5}
A.~Grothendieck, \emph{Cohomologie $\ell$-adique et fonctions ${L}$}, Lecture
  Notes in Mathematics, vol. 589, Springer-Verlag, 1977, S\'eminaire de
  G\'eom\'etrie Alg\'ebrique du Bois--Marie 1965--66 (SGA~5).

\bibitem[GS09]{GepSna}
D.~Gepner and V.~Snaith, \emph{On the motivic spectra representing algebraic
  cobordism and algebraic {$K$}-theory}, Doc. Math. \textbf{14} (2009),
  359--396. \MR{2540697}

\bibitem[ILO14]{Gabber}
L.~Illusie, Y.~Laszlo, and F.~Orgogozo (eds.), \emph{Travaux de {G}abber sur
  l'uniformisation locale et la cohomologie \'{e}tale des sch\'{e}mas
  quasi-excellents}, Soci\'{e}t\'{e} Math\'{e}matique de France, Paris, 2014,
  S\'{e}minaire \`a l'\'{E}cole Polytechnique 2006--2008. [Seminar of the
  Polytechnic School 2006--2008], With the collaboration of Fr\'{e}d\'{e}ric
  D\'{e}glise, Alban Moreau, Vincent Pilloni, Michel Raynaud, Jo\"{e}l Riou,
  Beno\^{i}t Stroh, Michael Temkin and Weizhe Zheng, Ast\'{e}risque No. 363-364
  (2014) (2014). \MR{3309086}

\bibitem[PW10]{PaninWalter1}
I.~{Panin} and C.~{Walter}, \emph{{On the motivic commutative ring spectrum
  BO}}, arXiv: 1011.0650, November 2010.

\bibitem[Rio10]{Riou}
J.~Riou, \emph{Algebraic {K}-theory, {$\mathbf A\sp 1$}-homotopy and
  {R}iemann-{R}och theorems}, J. Topol. \textbf{3} (2010), no.~2, 229--264.

\bibitem[R{\O}16]{RonOst}
O.~R\"{o}ndigs and P.~A. {\O}stv{\ae}r, \emph{Slices of hermitian {$K$}-theory
  and {M}ilnor's conjecture on quadratic forms}, Geom. Topol. \textbf{20}
  (2016), no.~2, 1157--1212. \MR{3493102}

\bibitem[Sch10]{Schlicht1}
M.~Schlichting, \emph{The {M}ayer-{V}ietoris principle for
  {G}rothendieck-{W}itt groups of schemes}, Invent. Math. \textbf{179} (2010),
  no.~2, 349--433. \MR{2570120}

\bibitem[Sch17]{Schlicht2}
\bysame, \emph{Hermitian {$K$}-theory, derived equivalences and {K}aroubi's
  fundamental theorem}, J. Pure Appl. Algebra \textbf{221} (2017), no.~7,
  1729--1844. \MR{3614976}

\bibitem[Tho84]{Thomason1}
R.~W. Thomason, \emph{Absolute cohomological purity}, Bull. Soc. Math. France
  \textbf{112} (1984), no.~3, 397--406. \MR{794741}

\end{thebibliography}

\end{document}